\documentclass[11pt,reqno]{amsart}
\usepackage[utf8]{inputenc}
\usepackage[T1]{fontenc}
\usepackage{lmodern}
\usepackage{amsfonts,amsthm,amsmath,amssymb}
\usepackage{graphicx}
\usepackage{enumerate}
\usepackage{hyperref}
\usepackage{color}
\usepackage{tikz-cd}
\usepackage[margin=1in]{geometry}
\usepackage[
    maxbibnames=99,
    backend=biber,
    style=alphabetic,
    sorting=nyt,
    giveninits=true
]{biblatex}
\DeclareFieldFormat{pages}{#1}
\renewbibmacro{in:}{%
  \ifentrytype{article}
    {}
    {\bibstring{in}%
     \printunit{\intitlepunct}}}
\DeclareFieldFormat
[article,inbook,incollection,inproceedings,patent,thesis,unpublished]
  {title}{\mkbibemph{#1}}
\DeclareFieldFormat{journaltitle}{#1\isdot}
\DeclareFieldFormat[article]{volume}{\mkbibbold{#1}}
\DeclareFieldFormat[article]{number}{\bibstring{number}\addnbspace #1}

\renewbibmacro*{journal+issuetitle}{%
  \usebibmacro{journal}%
  \setunit*{\addspace}%
  \iffieldundef{series}
    {}
    {\newunit
     \printfield{series}%
     \setunit{\addspace}}%
  \printfield{volume}%
  \setunit{\addspace}%
  \usebibmacro{issue+date}%
  \setunit{\addcomma\space}%
  \printfield{number}%
  \setunit{\addcolon\space}%
  \usebibmacro{issue}%
  \setunit{\addcomma\space}%
  \printfield{eid}
  \newunit}

\newtheoremstyle{mystyle}
  {}
  {}
  {\itshape}
  {}
  {\bfseries}
  {.}
  { }
  {\thmname{#1}\thmnumber{ #2}\thmnote{ (#3)}}

\theoremstyle{mystyle}
\newtheorem{Thm}{Theorem}[section]
\newtheorem{Lem}[Thm]{Lemma}
\newtheorem{Cor}[Thm]{Corollary}
\newtheorem{Prop}[Thm]{Proposition}

\theoremstyle{definition}

\theoremstyle{remark}
\newtheorem{Rmk}[Thm]{Remark}

\newcommand{\Z}{\mathbb{Z}}

\title{Deciding strong quasipositivity and quasipositivity}

\author{Marc Kegel}
\address{Universidad de Sevilla, Dpto.\ de Álgebra,
Avda.\ Reina Mercedes s/n,
41012 Sevilla, Spain}
\email{kegelmarc87@gmail.com}

\author{Qiuyu Ren}
\address{Department of Mathematics, Stanford University, Stanford, CA 94305, USA}
\email{qren18@stanford.edu}

\begin{document}

\begin{abstract}
We present two results concerning the decidability of (strong) quasipositivity of links in the 3-sphere. First, using a result of Dynnikov--Prasolov, we provide an algorithm to decide whether or not a link is strongly quasipositive. Second, we reduce the decision problem of the quasipositivity of links to that of braids.
\end{abstract}

\maketitle

\section{Introduction}
This article is motivated by the following decision problems.
\begin{itemize}
    \item Is quasipositivity for a braid decidable?
    \item Is strong quasipositivity for a braid decidable?
    \item Is quasipositivity for a link decidable?
    \item Is strong quasipositivity for a link decidable?
\end{itemize}
The precise version of these questions is the following. Does there exist an algorithm that takes as input a braid word in the Artin generators or a diagram of an oriented link and outputs whether or not the braid or the link is (strongly) quasipositive?

Of these four problems, strong quasipositivity for braids is known to be decidable, as we recall in Remark~\ref{rem:SQP_decidable_braid}. The focus of this article is on the two corresponding decision problems for links. We prove that strong quasipositivity of links is decidable, and reduce the quasipositivity problem for links to the quasipositivity problem for braids. The latter remains open in general.

\subsection{Deciding strong quasipositivity for links}

\begin{Thm}\label{thm:SQP_links}
There exists an algorithm that decides whether or not a given oriented link is strongly quasipositive.
\end{Thm}

A link is \textit{strongly quasipositive} if it has a quasipositive Seifert surface. These notions were introduced by Rudolph~\cite{Rudolph1990,Rudolph1992}. By the Bennequin inequality, a quasipositive Seifert surface always achieves maximal Euler characteristic among all Seifert surfaces of the given link. It is an open problem whether sharpness of the Bennequin inequality conversely implies strong quasipositivity (cf.~\cite[Problem~1.69]{K3}).

Theorem~\ref{thm:SQP_links} is a consequence of the following two assertions. The first is essentially a consequence of key results of Baader--Ishikawa~\cite{BaaderIshikawa2009}, who characterize quasipositive surfaces in terms of convex surface theory, and Dynnikov--Prasolov~\cite{DynnikovPrasolovDistinguishing}, who develop a combinatorial version of convex surface theory which makes the detection algorithmically possible.

\begin{Prop}\label{prop:QP_surfaces}
There is an algorithm that decides whether or not a given compact, oriented (not necessarily connected) surface in $S^3$ is quasipositive.\footnote{Here the input of a surface $F$ in $S^3$ is given as a triangulation of a pair $(S^3,F)$.}
\end{Prop}

For hyperbolic links, by a normal surface argument, there are only finitely many Seifert surfaces with maximal Euler characteristic, up to isotopy rel boundary, and Theorem~\ref{thm:SQP_links} for these links follows easily from Proposition~\ref{prop:QP_surfaces}. For non-hyperbolic links, this is not necessarily true (see e.g.~\cite{Eisner}), but we prove a similar finiteness result with the isotopy relaxed to free isotopies.

\begin{Prop}\label{prop:finite_Seifert_surfaces}
There is an algorithm that takes as input an oriented link $L$ and outputs a finite list (perhaps with repetitions) of all Seifert surfaces of $L$ (without closed components) with maximal Euler characteristic, up to free isotopy.
\end{Prop}

\subsection{Deciding quasipositivity for links}
In a slightly different direction, we can reduce the decidability of quasipositivity for a link to the decidability of quasipositivity for a braid. The following result is essentially a consequence of arguments in Hayden~\cite{hayden2018minimal} and Ito~\cite{Ito2022quantitative}.

\begin{Thm}\label{thm:finite_QP}
    If a link $L$ is quasipositive with braid index $b(L)$, then $L$ admits a quasipositive $b(L)$-braid representative of Artin word length at most
    \begin{align*}
       2b(L) \big( b(L)-\chi(L) \big),
    \end{align*}
where $\chi(L)$ denotes the maximal Euler characteristic of a Seifert surface of $L$ without closed components.
\end{Thm}

We remark that Ito's result~\cite{Ito2022quantitative} in fact yields a sharper but slightly more complicated upper bound on the length of a suitable quasipositive minimal-index representative. Since the braid index and the maximal Euler characteristic of a link are computable~\cite{Ito2022quantitative,HakenNormal}, cf.~\cite{TollefsonWang,AgolHassThurston}, we obtain the following corollary by enumerating all $b(L)$-braids with a bounded word length and with closure $L$. 

\begin{Cor}\label{cor:QP_links}
There exists an algorithm that takes as input an oriented link $L$ and outputs a finite list of braids $\beta_1,\ldots,\beta_N\in B_{b(L)}$, such that $L$ is quasipositive if and only if at least one $\beta_i$ is quasipositive.

In particular, quasipositivity for links is decidable if quasipositivity for braids is decidable.\qed
\end{Cor}

\begin{Rmk}\label{rem:SQP_decidable_braid}
Strong quasipositivity for a braid is decidable. Indeed, the Birman--Ko--Lee band generators define the so-called dual positive braid monoid~\cite{BirmanKoLee,BessisDual}. By definition, a braid $\beta$ is strongly quasipositive if and only if $\beta$ lies in the dual positive braid monoid. Moreover, the dual Garside normal form
$\beta=\delta^pA_1\cdots A_r$
is effectively computable, and $\beta$ is strongly quasipositive if and only if $p=\inf(\beta)\ge0$. Thus strong quasipositivity of a given braid is decidable.

On the other hand, it appears to be open whether quasipositivity for a braid is decidable. We remark that this would be a consequence of \cite[Problem~2.36]{K3}. For $3$-braids, quasipositivity is known to be decidable by work of Orevkov~\cite{OrevkovQP3}. More generally, Orevkov proved decidability for braids of arbitrary index and algebraic length two~\cite{OrevkovQP2}, and for $4$-braids of algebraic length three~\cite{OrevkovQP4}.
\end{Rmk}

\subsection*{Statement on AI use}

Parts of the mathematical arguments in this article were developed through interaction with GPT-5.6 Sol Pro. In detail, the proof of Theorem~\ref{thm:finite_QP} was developed by the model after we prompted it to combine the arguments of~\cite{hayden2018minimal} and~\cite{Ito2022quantitative}, while the proof of Proposition~\ref{prop:QP_surfaces} was obtained after we explained to it the case of annuli using results of Dynnikov--Prasolov \cite{DynnikovPrasolov} and suggested that a generalization might be possible.

The authors have reviewed the final mathematical arguments and take responsibility for their correctness. No presented text was created by an AI.

\subsection*{Acknowledgments}
We thank Sebastian Baader, Juan González-Meneses, Lukas Lewark, and Maxim Prasolov for feedback and useful discussion.

\subsection*{Individual grant support}
MK was supported by a Ram\'on y Cajal grant \mbox{(RYC2023-043251-I)} and PID2024-157173NB-I00 funded by MCIN/AEI/10.13039/501100011033, by ESF+, and by FEDER, EU; and by a VII Plan Propio de Investigación y Transferencia (SOL2025-36103) of the University of Sevilla. This research was conducted during the period when QR served as a Clay Research Fellow.

\section{Deciding strong quasipositivity}

For the proof of Proposition~\ref{prop:QP_surfaces}, we will use the characterization of quasipositive surfaces in the standard tight $S^3$ in terms of ribbons of Legendrian graphs, which we briefly recall.
For more details on ribbons of Legendrian graphs, we refer to~\cite{BaaderIshikawa2009,Avdek,HaydenLegendrianRibbons,contact_Kirby_moves}.

A \emph{Legendrian ribbon} of a Legendrian graph $\Lambda$ in $(S^3,\xi_{st})$ is a compact oriented surface $R$ in $(S^3,\xi_{st})$ containing $\Lambda$ such that there exist 
\begin{enumerate}[(a)]
    \item a contact form $\alpha$ for $\xi_{st}$ whose Reeb vector field $R_\alpha$ is positively transverse to $R$; and
    \item a vector field $V$ on $R$ lying in the characteristic foliation of $R$, which is positively transverse to $\partial R$ and whose negative flow $\Phi_{-t}^{V}$ contracts $R$ onto $\Lambda$, i.e.
    $$
    \bigcap_{t>0}\Phi_{-t}^{V}(R)=\Lambda.
    $$
\end{enumerate} 
Every Legendrian graph admits a Legendrian ribbon, which contains the graph as a spine. The following result due to Baader and Ishikawa~\cite{BaaderIshikawa2009} identifies Legendrian ribbons with quasipositive surfaces. 

\begin{Lem}[Baader--Ishikawa, Lemma 2.1 and Theorem 2.2 in~\cite{BaaderIshikawa2009}]\label{lem:QP_Legendrian_ribbon}
A compact oriented surface $F$ in $S^3$ is quasipositive if and only if it is isotopic (as an oriented surface) to a Legendrian ribbon in $(S^3,\xi_{{st}})$.\qed
\end{Lem}

We will also use the following consequence of the work of Dynnikov and Prasolov~\cite{DynnikovPrasolovDistinguishing}, regarding the algorithmic realization of a multicurve on a surface as the dividing set on a convex surface.
\begin{Lem}[Dynnikov--Prasolov~\cite{DynnikovPrasolovDistinguishing}]\label{lem:DP_dividing_set_algorithm} 
There is an algorithm which, given a closed oriented surface $\Sigma$ in $(S^3,\xi_{st})$ and an oriented $1$-manifold $\Gamma$ on $\Sigma$, decides whether or not the pair $(\Sigma,\Gamma)$ is isotopic to a convex surface and its dividing set.
\end{Lem} 

\begin{proof}
The statement is explicitly written in the third to last paragraph on Page~702 of \cite{DynnikovPrasolovDistinguishing}, although no explicit proof was given. The result follows essentially from Corollary~2.1 in~\cite{DynnikovPrasolovDistinguishing} and the discussion immediately following it by choosing in that corollary the link $L$ and its rectangular diagram $R$ to be empty.
\end{proof}

\begin{proof}[Proof of Proposition~\ref{prop:QP_surfaces}]
We assume $F$ has no closed components, since otherwise it is not quasipositive by definition.

We identify a tubular neighborhood $\nu(F)$ of $F$ with a product $[-1,1]\times F$. After smoothing the corners, its boundary $\partial\nu(F)$ is diffeomorphic to the double $D(F)$ of $F$. Since every component of $F$ has nonempty boundary, $\Gamma_F:=\{0\}\times\partial F$ divides $D(F)$ into two regions $F_\pm$ containing $\{\pm1\}\times F$. We orient $\Gamma_F$ as the boundary of $F_+$.

We claim that $F$ is quasipositive if and only if the pair $(D(F),\Gamma_F)$ is isotopic to a convex surface and its dividing set.

If $F$ is quasipositive, by Lemma~\ref{lem:QP_Legendrian_ribbon}, we may assume that it is a Legendrian ribbon. Taking $\nu(F)$ to be a small tubular neighborhood of $F$, we see from condition (b) of Legendrian ribbonness that the characteristic foliation on $D(F)$ is divided by $\Gamma_F$ in the sense of Giroux. Thus, $(D(F),\Gamma_F)$ is isotopic to a convex surface by Giroux (see \cite[Proposition~II.2.1]{Giroux1991}; the compatibility of the orientation on $\Gamma_F$ follows from the condition (a)). Conversely, if $D(F)$ is convex with dividing set $\Gamma_F$, by an isotopy, we may Legendrian realize any spine $\Lambda$ of $F_+$ by the Legendrian realization principle so that $F_+$ is a Legendrian ribbon of it; see, for example, \cite[Lemma~2.2]{BakerOnaran}. In particular, $F_+$, hence also $F$, is quasipositive by Lemma~\ref{lem:QP_Legendrian_ribbon}.

Finally, from a triangulation of the pair $(S^3,F)$ we can algorithmically construct a triangulation of the triple $(S^3,D(F),\Gamma_F)$. Applying the algorithm from Lemma~\ref{lem:DP_dividing_set_algorithm} to $(D(F),\Gamma_F)$ therefore decides whether or not $F$ is quasipositive.
\end{proof}

\begin{proof}[Proof of Proposition~\ref{prop:finite_Seifert_surfaces}]
We first reduce to the case when $L$ is nonsplit. Suppose $L=L_1\sqcup\cdots\sqcup L_k$ where each $L_i$ is nonempty and nonsplit and chosen to be contained in some ball $B_i\subset S^3$, where the $B_i$'s are pairwise disjoint. Let $S_i^{\alpha_i}$, $\alpha_i\in A_i$ be a list of all Seifert surfaces of $L_i$ of maximal Euler characteristic, up to free isotopy, all chosen to be contained in $B_i$, we claim that $S_1^{\alpha_1}\sqcup\cdots\sqcup S_k^{\alpha_k}$, $\alpha_i\in A_i$, is a list of all Seifert surfaces of $L$ of maximal Euler characteristic, up to free isotopy.

Suppose $S$ is a Seifert surface of $L$ with maximal Euler characteristic. By considering the intersection with the essential spheres encircling the $L_i$'s, it is easy to see that $S=S_1\cup\cdots\cup S_k$ where each $S_i\subset S^3$ is a Seifert surface for $L_i$ with maximal Euler characteristic. We induct on $j$ to show that $S$ is freely isotopic to a surface of the form $(S_1^{\alpha_1}\sqcup\cdots S_j^{\alpha_j})\cup S_{j+1}'\cup\cdots S_k'$ for some $\alpha_i\in A_i$, $1\le i\le j$. The case $j=0$ is trivial. Suppose $1\le j\le k$. We may assume $j=k$, since we can always forget $S_{j+1},\cdots,S_k$ by dragging them along during the isotopy. Thus, by assumption, $S$ is freely isotopic to a surface of the form $(S_1^{\alpha_1}\sqcup\cdots S_{k-1}^{\alpha_{k-1}})\cup S_k$. By an innermost argument, we may assume $S_k\cap B_i=\emptyset$ for $i<k$. Now, by assumption, we may isotope $S_k$ in $S^3$ to some surface $S_k^{\alpha_k}$, $\alpha_k\in A_k$, during which we drag along the balls $B_i$ together with the surfaces $S_i^{\alpha_i}$ inside, $i<k$. Then, we perform a further isotopy in the exterior of $S_k^{\alpha_k}$ to move the balls $B_i$ back to their original positions, $i<k$. This finishes the inductive step, hence the reduction to the nonsplit case.\smallskip

If $L$ is the unknot or a Hopf link, there is a unique Seifert surface with maximal Euler characteristic, up to free isotopy. From now on, we assume $L$ is a nonsplit link that is not the unknot or a Hopf link.

\textbf{Claim 1}: If $S$ and $S'$ are Seifert surfaces of $L$ that are related by twisting along a torus $T=T^2\subset S^3\backslash \nu L$, then they are freely isotopic in $S^3$.

Here, by twisting along $T$, we mean that there is a level-preserving self-diffeomorphism $\Phi$ of a closed neighborhood $I\times T$ of $T$ rel boundary, such that $S'$ is obtained from $S$ by cutting $I\times T$ out and regluing it back by $\Phi$. Regarding $\Phi$ as a diffeomorphism of $S^3$, it is isotopic to $\mathrm{id}_{S^3}$ since $\mathrm{Diff}^+(S^3)$ is connected \cite{cerf1968sur}. This shows $S'=\Phi(S)$ is freely isotopic to $S$, proving the claim.\smallskip

Let $T_1,\cdots,T_k$ be the list of all JSJ tori and boundary tori of $S^3\backslash \nu L$.

\textbf{Claim 2}: One can algorithmically write down a (finite) list of isotopy classes of multicurves on $T_1\cup\cdots\cup T_k$ such that for any Seifert surface $S$ of $L$ with maximal Euler characteristic that is in minimal position with $T_1\cup\cdots\cup T_k$, their intersection lies in one of the classes in the list.

\textbf{Claim 3}: For any JSJ piece $M$ of $S^3\backslash \nu L$, any multicurve $C\subset\partial M$, and any integer $\chi_0$, there is an algorithm to find all properly embedded essential surfaces $S_0\subset M$ without closed components, with $\partial S_0=C$ and $\chi(S_0)\ge\chi_0$, up to twisting along incompressible tori.

Note that if $S$ is a Seifert surface of $L$ with maximal Euler characteristic that is in minimal position with respect to $T_1\cup\cdots\cup T_k$, then the part of $S$ inside any JSJ piece of $S^3\backslash \nu L$ is essential and has Euler characteristic bounded from above and (hence) from below. Therefore, Proposition~\ref{prop:finite_Seifert_surfaces} is a consequence of Claims 1,2,3. It remains to prove Claims 2 and 3.\smallskip

If $M$ is a hyperbolic piece in $S^3\backslash \nu L$, since the only incompressible $\partial$-incompressible surfaces with nonnegative Euler characteristic are the spheres, disks, and $\partial$-parallel tori, a normal surface theory argument (see e.g. \cite[Theorem~6.12]{Lackenby2022algorithms}) shows that there is an algorithmically effective finite list of essential surfaces in $M$ with Euler characteristic bounded from below by a prescribed constant, up to twisting along the boundary tori. This proves Claim 3 for the piece $M$ as well as Claim 2 for all boundary tori of $M$.

To prove Claims 2 and 3 in their full generality, we need to analyze the Seifert fibered pieces in the JSJ decomposition of link complements. By the classification of Seifert fibered submanifolds of $S^3$ (see e.g.~\cite[Proposition~4]{budney2006jsj}), every Seifert fibered piece in the JSJ decomposition of $S^3\backslash \nu L$ is one of the following, as a submanifold of $S^3$:
\begin{enumerate}
\item The complement of a $T(np,nq)$ torus link, $\mathrm{gcd}(p,q)=1$, $|p|,|q|>1$, $n\ge1$.
\item The complement of a $C(np,nq)$ cable link in a solid torus in $S^3$, $\mathrm{gcd}(p,q)=1$, $|p|>1$, $n\ge1$.
\item Some submanifold $P_n\times S^1$ of $S^3$, $n\ge3$, where $P_n$ is the connected planar surface with $n$ boundary components.
\end{enumerate}
Each of these manifolds admits a unique Seifert fibration. The base orbifolds are
\begin{enumerate}
\item $P_n(|p|,|q|)$,
\item $P_{n+1}(|p|)$,
\item $P_n$,
\end{enumerate}
respectively. Note that each of these orbifolds has negative Euler characteristic.

In a connected Seifert fibered manifold $M$ over an orbifold $\mathcal O$, essential surfaces are exactly the horizontal surfaces and the vertical surfaces without $\partial$-parallel components. Here, we recall that a surface $S_0\subset M$ is \textit{horizontal} if $S_0\to\mathcal O$ is a cyclic cover, and \textit{vertical} if it is a union of some regular fibers. A (marked) horizontal surface $S_0\subset M$ is determined by the isomorphism type of the cyclic cover $S_0\to\mathcal O$ together with an element in a lattice affine over $H^1(|\mathcal O|)$, where $|\mathcal O|$ is the underlying smooth surface of $\mathcal O$; this is because any two horizontal surfaces $S_0,S_0'\subset M$ inducing the same cyclic cover over $\mathcal O$ differ by a monodromy map $\pi_1^{orb}(\mathcal O)\to\Z$, which amounts to a class in $H^1(|\mathcal O|)$. (Essential) vertical surfaces are in one-to-one correspondence with (essential) multicurves (with arc components allowed) in $\mathcal O$.

In our setup, if $M$ is a Seifert fibered piece in $S^3\backslash \nu L$, since the base orbifold of $M$ has negative Euler characteristic, every horizontal essential surface in $M$ with Euler characteristic bounded from below has bounded covering degree over $\mathcal O$. In particular, if $T\subset\partial M$ is a boundary component and $S\subset M$ is an essential surface with Euler characteristic bounded from below, the geometric intersection between $S\cap T$ and the fiber slope on $T$ is bounded from above.

\begin{proof}[Proof of Claim 2]\hfill

\textbf{Case 1}: $T_i$ is adjacent to a hyperbolic piece.

Then, $S\cap T_i$ has finitely many possibilities that can be enumerated algorithmically by the discussion of the hyperbolic case above.

\textbf{Case 2}: $T_i$ is adjacent to two Seifert fibered pieces.

Then, the two adjacent Seifert fibered pieces have distinct fiber slopes $s_1,s_2$ on $T_i$. The argument above shows that $S\cap T_i$ has (effectively) bounded geometric intersection with both $s_1$ and $s_2$. This enables us to write down a finite list of all possible isotopy classes of $S\cap T_i$.

\textbf{Case 3}: $T_i$ is a boundary torus of $S^3\backslash \nu L$ that is adjacent to a Seifert fibered piece $M$.

Since $S$ is a Seifert surface of $L$, it intersects the meridian slope on $T_i$ geometrically once. If the meridian slope on $T_i$ induced from $L$ is different from the fiber slope on $T_i$ induced from $M$, the same argument as in Case 2 proves the claim.

By the description above of Seifert fibered pieces that could possibly appear, we see that the meridian slope on $T_i$ is different from the fiber slope unless $M=P_n\times S^1$ for some $n\ge3$. In the exceptional case $M=P_n\times S^1$, $M$ has at most one boundary component that is a boundary of $S^3\backslash \nu L$ on which the meridian agrees with the fiber, since Dehn-filling two components of $P_n\times S^1$ by the fiber slopes yields a manifold with a nonseparating $S^2$. If $T_i$ is such a boundary component, since the possible intersection $S\cap T_j$ for all $T_j\subset M$ with $j\ne i$ can be enumerated, the possible homology classes of $S\cap T_i$ can be enumerated. Since $S$ is a Seifert surface, $S\cap T_i$ has no algebraically canceling components, hence its possible isotopy classes can be enumerated, proving the claim.
\end{proof}

\begin{proof}[Proof of Claim 3]
The case when $M$ is hyperbolic was dealt with before. Assume therefore $M$ is Seifert fibered over some orbifold $\mathcal O$.

\textbf{Step 1}: Enumerate horizontal surfaces $S_0\subset M$ with $\chi(S_0)\ge\chi_0$, up to twisting along embedded vertical tori.

The covering degree of $S_0\to\mathcal O$ is bounded above by $\chi_0/\chi(\mathcal O)$. We may enumerate all such cyclic coverings by enumerating homomorphisms $\pi_1^{orb}(\mathcal O)\to\Z/n$ for all $1\le n\le\lfloor\chi_0/\chi(\mathcal O)\rfloor$. For each fixed isomorphism type of $S_0\to\mathcal O$, there are effectively finitely many covering isomorphisms $\partial S_0\cong C$ up to equivalence. For each such chosen isomorphism, all horizontal surfaces either form an empty set, or form a torsor over $H^1(|\mathcal O|,\partial|\mathcal O|)=H_1(|\mathcal O|)$. In the latter case, $H_1(|\mathcal O|)$ acts on the set of horizontal surfaces by torus twists along vertical tori, and an element of the torsor can be written down algorithmically. 

\textbf{Step 2}: Enumerate essential vertical surfaces in $M$ without closed components, up to twisting along embedded vertical tori.

If $C$ is not a union of fibers, there is no such vertical surface. Suppose $C$ is a union of fibers that projects down to some finite collection $P$ of points on $\partial\mathcal O$. Then, essential vertical surfaces in $M$ without closed components with boundary $C$ are in one-to-one correspondence with essential multiarcs in $\mathcal O$ bounding $P$. If $\alpha\subset\mathcal O$ is a multiarc bounding $P$ and $A_\alpha$ is the corresponding vertical surface, twisting $A_\alpha$ positively/negatively once along the vertical torus $T_\gamma$ corresponding to a simple closed curve $\gamma\subset\mathcal O$ yields the vertical surface $A_{\tau_\gamma^{\pm1}(\alpha)}$, where $\tau_\gamma$ is the Dehn twist along $\gamma$. Since $\mathcal O\backslash\alpha$ has finitely many possible topological types, there is an algorithmically effective finite collection of such $\alpha$ up to the mapping class group action of $\mathcal O$. Equivalently, since Dehn twists generate mapping class group, there is an algorithmically effective finite collection of essential vertical surfaces in $M$ bounding $C$ up to twisting along embedded vertical tori.
\end{proof}

The proof of Proposition~\ref{prop:finite_Seifert_surfaces} is complete.
\end{proof}

\begin{Rmk}
In fact, the proof of Proposition~\ref{prop:finite_Seifert_surfaces} gives an algorithm to enumerate all incompressible Seifert surfaces with Euler characteristic bounded from below by any given bound, up to free isotopy, of any given link.
\end{Rmk}

\begin{proof}[Proof of Theorem~\ref{thm:SQP_links}]
A quasipositive Seifert surface $\Sigma$ of $L$ achieves equality in the Bennequin inequality $\overline{sl}(L)\le-\chi(\Sigma)$ \cite{Bennequin1983}, where $\overline{sl}$ denotes the maximal self-linking number. Therefore, a quasipositive Seifert surface of $L$ maximizes the Euler characteristic among Seifert surfaces of $L$. By Proposition~\ref{prop:finite_Seifert_surfaces}, we can enumerate all Seifert surfaces of $L$ that maximize the Euler characteristic, up to free isotopy. By Proposition~\ref{prop:QP_surfaces}, we can decide whether or not there is a Seifert surface among this finite list that is quasipositive, and thereby decide whether or not $L$ is strongly quasipositive.
\end{proof}

\section{Deciding quasipositivity}
\begin{proof}[Proof of Theorem~\ref{thm:finite_QP}]
Hayden~\cite{hayden2018minimal} shows that every quasipositive link has a minimal-index braid representative that is quasipositive. We only need to control the word length of such a representative.

Building on arguments in Birman--Menasco~\cite{birman1992studying}, the proof of Ito~\cite[Theorem~1.3]{Ito2022quantitative} shows that every minimal-index braid representative of an oriented link $L$ is related to such a representative with word length at most $2b(L)(b(L)-\chi(L))$ via a sequence of conjugations and exchange moves. Here, an exchange move in $B_n$ is the move relating $\sigma_{n-1}X\sigma_{n-1}^{-1}Y$ to $\sigma_{n-1}Y\sigma_{n-1}^{-1}X$ for some $X,Y$ in the standard subgroup $B_{n-1}$ of $B_n$. Conjugations preserve quasipositivity of braids, as do exchange moves by work of Orevkov~\cite{orevkov2000markov} and Birman--Wrinkle~\cite[Figure~8]{birman2000transversally}, as observed explicitly by Hayden~\cite[proof of Theorem~1.2]{hayden2018minimal}. Thus \cite{Ito2022quantitative} gives the desired bound on the word length.
\end{proof}

\printbibliography

\end{document}